\documentclass[12pt]{amsart}
\usepackage{graphics, amsmath, amsfonts, amssymb, amscd, mathrsfs}

\input{xy}
\xyoption{all}

\newtheorem{theorem}{Theorem}
\newtheorem{proposition}[theorem]{Proposition}
\newtheorem{lemma}[theorem]{Lemma}
\newtheorem{corollary}[theorem]{Corollary}

\theoremstyle{definition}
\newtheorem{definition}[theorem]{Definition}
\newtheorem{example}[theorem]{Example}
\newtheorem{remark}[theorem]{Remark}

\newcommand{\C}{\mathbb{C}}
\newcommand{\D}{\mathbb{D}}
\renewcommand{\P}{\mathbb{P}}
\newcommand{\T}{{\mathbb T}}
\newcommand{\Z}{{\mathbb Z}}
\renewcommand{\Re}{\operatorname{Re}}

\title{Holomorphic flexibility properties \\ of compact complex surfaces}

\author{Franc Forstneri\v c and Finnur L\'arusson}

\address{Franc Forstneri\v c, Faculty of Mathematics and Physics, University of Ljubljana, and Institute of Mathematics, Physics, and Mechanics, Jadranska 19, 1000 Ljubljana, Slovenia}
\email{franc.forstneric@fmf.uni-lj.si}

\address{Finnur L\'arusson, School of Mathematical Sciences, University of Adelaide, Adelaide SA 5005, Australia}
\email{finnur.larusson@adelaide.edu.au}

\thanks{Forstneri\v c was supported by grant P1-0291 from ARRS, Republic of Slovenia.  L\'arusson was supported by Australian Research Council grant DP120104110.}



\date{20 July 2012.  Most recent changes 21 February 2013}

\begin{document}

\begin{abstract}  We introduce the notion of a stratified Oka manifold and prove that such a manifold $X$ is strongly dominable in the sense that for every $x\in X$, there is a holomorphic map $f:\C^n\to X$, $n=\dim X$, such that $f(0)=x$ and $f$ is a local biholomorphism at $0$.  We deduce that every Kummer surface is strongly dominable.  We determine which minimal compact complex surfaces of class VII are Oka, assuming the global spherical shell conjecture.  We deduce that the Oka property and several weaker holomorphic flexibility properties are in general not closed in families of compact complex manifolds.  Finally, we consider the behaviour of the Oka property under blowing up and blowing down.

\end{abstract}

\maketitle
\tableofcontents


\section{Introduction}

\noindent
The class of Oka manifolds has emerged from the modern theory of the Oka principle, initiated in 1989 in a seminal paper of Gromov \cite{Gromov}.  They were first formally defined by the first-named author in 2009 in the wake of his result that some dozen possible definitions are all equivalent \cite{Forstneric2009}.  A complex manifold $X$ is said to be an \textit{Oka manifold} if the homotopy principle holds for maps from Stein sources into $X$, meaning that every continuous map from a Stein manifold (or, more generally, a reduced Stein space) $S$ into $X$ can be deformed to a holomorphic map, with interpolation on a closed complex subvariety of $S$, uniform approximation on a holomorphically convex compact subset of $S$, and with continuous dependence on a parameter.  Equivalently, for every $n\geq 1$, every holomorphic map from an open neighbourhood of a convex compact subset $K$ of $\C^n$ to $X$ can be uniformly approximated on $K$ by entire maps $\C^n\to X$.  This property, the \textit{convex approximation property,} is the weakest version of the Oka property.  Looking at the Oka property formulated this way, it is immediate that it passes up and down any holomorphic covering map.  More generally, and this is not quite as easy to see, the Oka property passes up and down in any holomorphic fibre bundle with Oka fibres.

The Oka property can be seen as an answer to the question: what should it mean for a complex manifold to be \lq\lq anti-hyperbolic\rq\rq?  Gromov's Oka principle is about sufficient geometric conditions for the Oka property to hold.  The most important such condition is \textit{ellipticity}, that is, possessing a \textit{dominating spray}, a structure that generalises the exponential map of a complex Lie group (\cite{Gromov}, Section 0.5).  For more background, see the monograph \cite{Forstneric2011} and the survey \cite{Forstneric-Larusson}.

One of the central problems of Oka theory is to determine the place of Oka manifolds in the classification of compact complex manifolds.  This is well understood only for manifolds of dimension $1$: a Riemann surface (whether compact or not) is Oka if and only if it is not Kobayashi hyperbolic, and this holds if and only if its universal covering space is not the disc.  In particular, the compact Riemann surfaces that are Oka are the Riemann sphere and all elliptic curves.  Already for complex surfaces the problem is difficult and to a large extent open.  Whether the Oka property is preserved by blowing up and blowing down is a closely related problem, also difficult and very much open.  In particular, we do not know whether an arbitrary Oka manifold blown up at a point is still Oka.  This paper is a contribution towards a solution to these two problems.

We will be concerned with several properties of complex manifolds that are (at least ostensibly) weaker than the Oka property.  

\begin{definition}  \label{d:flexibility-properties}
Let $X$ be a complex manifold (here always taken to be connected).

(a)  $X$ is \textit{stratified Oka} if it admits a stratification $X=X_0\supset X_1\supset\cdots\supset X_m=\varnothing$ by closed complex subvarieties, such that every connected component (stratum) of each difference $X_{j-1}\setminus X_j$, $j=1,\ldots,m$, is an Oka manifold.

(b)  $X$ is \textit{dominable at a point} $x\in X$ (\textit{by} $\C^n$, where $n=\dim X$) if there is a holomorphic map $f:\C^n\to X$ such that $f(0)=x$ and $f$ is a local biholomorphism at $0$.

(c)  $X$ is \textit{dominable} if it is dominable at some point.

(d)  $X$ is \textit{strongly dominable} if it is dominable at every point.

(e)  $X$ is $\C$-\textit{connected} if any two points in $X$ can be joined by a finite chain of entire curves in $X$.  (This definition has several variants: see Remark~\ref{r:C-connectedness}.)
 
(f)  $X$ is \textit{strongly Liouville} if the universal covering space of $X$ carries no nonconstant negative plurisubharmonic functions. 
\end{definition}

Properties (a), (c), (d), and (e) are anti-hyperbolic in the sense that the only Kobayashi hyperbolic manifold satisfying any of them is the point.  (For the weakest property (f), this fails.  A simply connected compact Kobayashi hyperbolic manifold, such as a smooth Kobayashi hyperbolic surface in $\P_3$, satisfies (f).)  We refer to these properties, the Oka property, ellipticity, and other similar properties as \textit{holomorphic flexibility} properties to contrast them with the rigidity that characterises hyperbolicity.  (A more specific definition of flexibility exists in the literature (\cite{Forstneric2011}, Definition 5.5.16), but we shall not use it here.)

An Oka manifold is obviously stratified Oka; the converse is open.  The following implications are easily verified.
\begin{itemize}
\item  If $X$ is Oka, then $X$ is strongly dominable.
\item  If $X$ is strongly dominable, then $X$ is $\C$-connected.
\item  If $X$ is stratified Oka, then $X$ is dominable.
\item  If $X$ is either dominable or $\C$-connected, then $X$ is strongly Liouville.  
\end{itemize}

The last implication depends on the well-known fact that $\C^n$ carries no nonconstant negative plurisubharmonic functions.  It follows in particular that a manifold that is not strongly Liouville is also not Oka, an observation that will be used in the sequel.

The first main result of the paper is the following.

\begin{theorem}  \label{t:stratified-Oka}
A stratified Oka manifold is strongly dominable.
\end{theorem}

A Kummer surface $X$ admits a stratification $X\supset C\supset\varnothing$, where $C$ is the union of 16 mutually disjoint smooth rational curves, and the difference $X\setminus C$ is an Oka manifold (Lemma~\ref{l:Kummer-are-stratified-Oka}).  Thus $X$ is stratified Oka, and we obtain the following corollary.

\begin{corollary}  \label{c:Kummer-strongly-dominable}
Every Kummer surface is strongly dominable.
\end{corollary}

Kummer surfaces are dense in the moduli space of all K3 surfaces, but we do not know whether it follows that all K3 surfaces are strongly dominable.  In fact, we prove that strong dominability is in general not closed in families of compact complex manifolds (Corollary \ref{c:not-closed}).

Class VII in the Enriques-Kodaira classification comprises the non-algebraic compact complex surfaces of Kodaira dimension $\kappa=-\infty$.  Minimal surfaces of class VII fall into several mutually disjoint classes.  For second Betti number $b_2=0$, we have Hopf surfaces and Inoue surfaces.  For $b_2\geq 1$, there are Enoki surfaces, Inoue-Hirzebruch surfaces, and intermediate surfaces; together they form the class of Kato surfaces.  By the \textit{global spherical shell conjecture}, currently proved only for $b_2=1$ by Teleman \cite{Teleman}, every minimal surface of class VII with $b_2\geq 1$ is a Kato surface.  Assuming that the conjecture holds, we determine which minimal surfaces of class VII are Oka.

\begin{theorem}  \label{t:class-VII}
Minimal Hopf surfaces and Enoki surfaces are Oka.  Inoue surfaces, Inoue-Hirzebruch surfaces, and intermediate surfaces, minimal or blown up, are not strongly Liouville, and hence not Oka. 
\end{theorem}

Enoki surfaces are generic among Kato surfaces.  Inoue-Hirzebruch surfaces and inter\-mediate surfaces can be obtained as degenerations of Enoki surfaces (explicit examples are given in \cite{Dloussky}).  Thus, Theorem \ref{t:class-VII} yields the following corollary.

\begin{corollary}  \label{c:not-closed}
Compact complex surfaces that are Oka can degenerate to a surface that is not strongly Liouville.  Consequently, the following properties are in general not closed in families of compact complex manifolds.
\begin{itemize}
\item  The Oka property.
\item  The stratified Oka property.
\item  Strong dominability.
\item  Dominability.
\item  $\C$-connectedness.
\item  Strong Liouvilleness.
\end{itemize}
\end{corollary}

The corollary answers a question posed in \cite{Larusson2012}.  There it was shown that the Oka fibres in a family of compact complex manifolds form a $G_\delta$ set.  The corollary says that the set need not be closed.  In fact, the corollary suggests that the only interesting closed anti-hyperbolicity property is the weakest anti-hyperbolicity property, the property of not being Kobayashi hyperbolic.  Now the open question is whether the set of Oka fibres in a family is open, that is, whether the Oka property is stable under small deformations.

Here is what we know about which minimal compact complex surfaces are Oka.
\begin{itemize}
\item  $\kappa=-\infty$:  Rational surfaces are Oka.  A ruled surface is Oka if and only if its base is Oka.  Theorem~\ref{t:class-VII} covers class VII if the global spherical shell conjecture is true.
\item  $\kappa=0$:  Bielliptic surfaces, Kodaira surfaces, and tori are Oka.  It is unknown whether any or all K3 surfaces or Enriques surfaces are Oka.
\item  $\kappa=1$:  Buzzard and Lu determined which properly elliptic surfaces are dominable \cite{Buzzard-Lu}.  Nothing further is known about the Oka property for these surfaces.
\item  $\kappa=2$:  Surfaces of general type are not dominable (this is an easy consequence of \cite{Kobayashi-Ochiai}, Theorem~2), and hence not Oka.
\end{itemize}

In the final section, we show that often, something survives of the Oka property when an Oka manifold is blown down (Theorem~\ref{t:blow-down}).  We also establish a perhaps surprising consequence of the hypothesis that the Oka property is preserved by blow-ups, which may suggest that the hypothesis is false (Proposition~\ref{p:prediction}).

\smallskip\noindent
\textit{Acknowledgement.}  We are grateful to the late Marco Brunella for drawing our attention to the relevance of surfaces of class VII to the question of whether the Oka property is closed in families.  We thank Georges Dloussky for help with the theory of surfaces of class VII.


\section{Stratified Oka manifolds are strongly dominable}  \label{s:stratified-Oka}

In this section we prove Theorem~\ref{t:stratified-Oka}.  The proof is based on the following result, which is a special case of \cite{Forstneric2011}, Theorem 7.6.1.

\begin{theorem}   \label{t:Hayama}
Let $X$ be a Stein manifold, $Y$ be a complex manifold, $Y'$ be a closed complex subvariety of $Y$, and $f: X\to Y$ be a continuous map that is holomorphic on an open neighbourhood of the preimage $X'=f^{-1}(Y')$ in $X$.  If $Y\setminus Y'$ is an Oka manifold, then for each $k\geq 1$, there exists a homotopy of continuous maps $f_t: X\to Y$, $t\in [0,1]$, such that $f=f_0$, $f_1$ is holomorphic on $X$, and for each $t\in [0,1]$, $f_t$ is holomorphic near $X'$, agrees with $f=f_0$ to order $k$ along $X'$, and maps $X\setminus X'$ into $Y\setminus Y'$, that is, $f_t^{-1}(Y')=X'$.
\end{theorem}

Note that $X'=f^{-1}(Y')$ is a closed complex subvariety of $X$ since $f$ is assumed to be holomorphic on a neighbourhood of $X'$.  

The special case of Theorem \ref{t:Hayama} when $Y'=\{0\}\subset \C^d=Y$ follows from classical Oka-Grauert theory since in this case $Y\setminus Y'= \C^d\setminus \{0\}$ is complex homogeneous; it corresponds to the problem of \textit{complete intersections} (see  \cite{Forstneric2011}, Section 7.5). The general case, similar to the one stated here, was first proved in \cite{Forstneric2001} (Theorem 1.3). The proof of the general case is explained in \cite{Forstneric2011}, Section 7.6, and is based on the paper \cite{Forstneric2010}. 

\begin{proof}[Proof of Theorem~\ref{t:stratified-Oka}]  
Let $Y$ be a stratified Oka manifold of dimension $n$, and let $Y=Y_0\supset Y_1\supset\cdots\supset Y_l=\varnothing$ be a stratification of $Y$ such that $S_j=Y_j\setminus Y_{j+1}$ is Oka for $j=0,\ldots,l-1$.  Given a point $y_0\in Y$, we wish to find a holomorphic map $F:  \C^n\to Y$ such that $F(0)=y_0$ and $F$ is dominating at $0$.

If $y_0\in S_0=Y_0\setminus Y_1$, then the Oka property of $S_0$ implies that there is a holomorphic map $F: \C^n \to S_0$ with $F(0)=y_0$ and rank $n$ at $0$. 

Now suppose $y_0\in S_1=Y_1\setminus Y_2$.  Let $\Sigma$ be the connected component of $S_1$ containing $y_0$.  Set $m=\dim\Sigma$ and $d=n-m$.  Since $\Sigma$ is Oka by assumption, there is a holomorphic map $g:\C^m\to \Sigma$ such that $g(0)=y_0$ and $g$ is a local biholomorphism at $0$.  

We identify $\C^m$ with the subspace $\C^m\times\{0\}^d$ of $\C^n$.  Write the coordinates on $\C^n$ as $z=(z',z'')$ with $z'=(z_1,\ldots,z_m)$ and $z''=(z_{m+1},\ldots,z_n)$.  We shall construct a holomorphic map $F:\C^n \to Y$ such that $F(z',0)=g(z')$ for all $z'\in\C^m$, $F$ maps $\C^n\setminus (\C^m\times\{0\}^d)$ into $S_0$, and $F$ is a local biholomorphism at $0$.

Let $N\to\Sigma$ be the holomorphic normal bundle of $\Sigma$ in $Y$.  By Grauert's Oka principle, the pullback $g^*N\to \C^m$ is a trivial holomorphic vector bundle over $\C^m$  (see \cite{Forstneric2011}, Section 7.2).  A trivialisation of this bundle is given by $d$ linearly independent holomorphic vector fields $V_1,\ldots,V_d$ on $Y$ along $g$ that are normal to $\Sigma$.  More precisely, for every point $z'\in \C^m$ and its image $w'=g(z')\in\Sigma$, the vectors $V_1(z'),\ldots,V_d(z') \in T_{w'} Y$ give a basis of the normal space $N_{w'} = T_{w'} Y / T_{w'} \Sigma$.  

The graph 
\[ G=\{(z',g(z')): z'\in \C^m\}  \subset \C^m\times Y \]
of $g$ is a submanifold of $\C^m\times Y$, biholomorphic to $\C^m$, so it has a Stein open neighbourhood $\Omega$ in $\C^m\times Y$ by Siu's theorem (\cite{Siu1976}; \cite{Forstneric2011}, Theorem 3.1.1).  We identify each $V_j$ with a vector field on $\C^m\times Y$, defined along $G$, that is tangential to the fibres of the projection $\pi_1: \C^m\times Y\to\C^m$. After shrinking $\Omega$ around $G$ if necessary, we can assume that $V_1,\ldots,V_d$ extend to holomorphic vector fields on $\Omega$ that are tangential to the fibres of $\pi_1$. (Since $V_1,\ldots,V_d$ are sections of a holomorphic vector bundle, we can also appeal to Cartan's Theorem B to extend them holomorphically to all of $\Omega$.)  Let $\phi^j_t$ denote the flow of $V_j$ for small complex values of time $t$ (depending on the initial point).  Let $\pi_2 : \C^m\times Y\to Y$ be the projection onto the second factor.  The formula
\[ f(z',z'')= \pi_2 \circ \phi^1_{z_{m+1}} \circ \cdots \circ \phi^d_{z_n} (z',g(z')) \]
defines a holomorphic map $f$ into $Y$ from an open neighbourhood of $\C^m\times\{0\}^d$ in $\C^n$, such that $f(z',0'')=g(z')$ for $z'\in\C^m$ and  
\[ \dfrac{\partial}{\partial z_j}f(z',z'') \bigg\vert_{z''=0}  = V_{j-m}(z') \quad\hbox{for} \ z'\in\C^m \ \hbox{and} \ j=m+1,\ldots,n. \]
Hence the differential $df_{(z',0'')}: T_{(z',0'')}\C^n \to T_{g(z')} Y$ is an isomorphism for every $z'\in \C^m$ near the origin. In particular, $f$ is dominating at $0\in\C^n$ and $f(0)=g(0')=y_0$. Furthermore, as the vector fields $V_1,\ldots,V_d$ trivialise the normal bundle to $\Sigma$ in $Y$, the above implies that there is a neighbourhood $U$ of $\C^m\times\{0\}^d$ in $\C^n$ such that 
\[ f(U\setminus \C^m\times\{0\}^d) \subset S_0=Y_0\setminus Y_1.\] 

We may contract $\C^n$ into $U$ by a smooth contraction that equals the identity on a smaller open neighbourhood $V\subset U$ of $\C^m\times\{0\}^d$.  Precomposing $f$ with this contraction yields a continuous map $\C^n\to Y$ which agrees with $f$ on $V$ and maps $\C^n\setminus (\C^m\times\{0\}^d)$ into $S_0$.  Theorem~\ref{t:Hayama} now provides an entire map $F:\C^n\to Y$ which agrees with $f$ to second order along $\C^m\times\{0\}^d$ and maps $\C^n\setminus (\C^m\times\{0\}^d)$ into $S_0$. In particular, $F$ is dominating at $0$ and $F(0)=y_0$.  This completes the proof when $y_0\in S_1$. 

In general, if $y_0\in S_k=Y_k\setminus Y_{k+1}$ for some $k\in\{1,\ldots,l-1\}$, we choose strata $\Sigma_j\subset S_j$ for $j=0,\ldots,k$ such that $y_0\in \Sigma_k$ and $\Sigma_j\subset \overline{\Sigma}_{j-1}$ for $j=1,\ldots,k$.  Let $m_j=\dim\Sigma_j$, so $m_0=n>m_1>\cdots >m_k$.  Set $d_j=m_j-m_{j+1}$ for $j=0,\ldots,k-1$.  Since $\Sigma_k$ is Oka, there is an entire map $g_k: \C^{m_k}\to\Sigma_k$ which sends $0$ to $y_0$ and is dominating at $0$.  By the above argument and downward induction over $j=0,\ldots,k-1$, there are entire maps $g_j: \C^{m_j} \to\Sigma_j$ such that $g_j=g_{j+1}$ on $\C^{m_{j+1}} \times\{0\}^{d_j}$, and $g_j$ is dominating at $0$ (as a map into $\Sigma_j$).  For $j=0$, we thus get an entire map $F=g_0: \C^n \to Y$ which is dominating at $0$ with $F(0)=y_0$.
\end{proof}


\section{Kummer surfaces are strongly dominable}  \label{s:Kummer-surfaces}

In this section we show that all Kummer surface are stratified Oka, and hence strongly dominable by Theorem \ref{t:stratified-Oka}. We also prove a variant of the Oka property for maps of Stein surfaces to Kummer surfaces.

Let us recall the structure of Kummer surfaces (see \cite{Barth-et-al} for more information).  Let $\T$ be a complex 2-torus, the quotient of $\C^2$ by a lattice $\Z^4\cong\Gamma\subset \C^2$ of rank $4$, acting on $\C^2$ by translations.  Let $\pi: \C^2\to\T=\C^2/\Gamma$ be the quotient map.  The involution $\C^2\to \C^2$, $(z_1,z_2)\mapsto (-z_1,-z_2)$, descends to an involution $\sigma:\T\to\T$ with precisely 16 fixed points $p_1,\ldots,p_{16}$.  In fact, if $\omega_1,\ldots,\omega_4\in\C^2$ are generators for $\Gamma$, then $p_1,\ldots,p_{16}$ are the images under $\pi$ of the 16 points $c_1\omega_1+\cdots+c_4\omega_4$, where $c_1,\ldots,c_4\in\{0,\tfrac 1 2\}$.  The quotient space $\T/\{1,\sigma\}$ is a 2-dimensional complex space with 16 singular points $q_1,\ldots,q_{16}$.  The singularities can be resolved by blowing up $q_1,\ldots,q_{16}$, yielding a smooth compact surface $Y$ containing 16 mutually disjoint smooth rational curves $C_1,\ldots,C_{16}$.  This is the Kummer surface associated to the torus $\T$ or to the lattice $\Gamma$.

Here is an alternative description.  Let $X$ denote the surface obtained by blowing up
the torus $\T$ at each of the 16 points $p_1,\ldots,p_{16}$.  Let $E_j\cong\P_1$ denote the exceptional divisor over $p_j$.  The involution $\sigma$ of $\T$ lifts to an involution $\tau: X \to X$ with the fixed point set $E=E_1\cup\cdots\cup E_{16}$.  The eigenvalues of the differential $d\tau$ at any point of $E$ are $1$ and $-1$.  Hence the quotient $X/\{1,\tau\}$ is smooth and contains 16 rational $(-2)$-curves $C_j\cong\P_1$, the images of the rational $(-1)$-curves $E_j$ in $X$.  The quotient is the Kummer surface $Y$.  Denoting by $\widehat \C^2$ the surface obtained by blowing up $\C^2$ at every point of the discrete set $\widetilde \Gamma= \pi^{-1}(\{p_1,\ldots,p_{16}\})$, we have the following diagram (see \cite{Barth-et-al}, p.\ 224).
\[ \xymatrix{ \widehat\C^2 \ar[r] \ar[d] & X \ar[r] \ar[d]  & Y \ar[d] \\ \C^2 \ar[r]^{\pi} & \T \ar[r] & \T/\{1,\sigma\}} \]

\begin{lemma}  \label{l:Kummer-are-stratified-Oka}
Let $Y$ be a Kummer surface with the exceptional rational curves $C_1,\ldots,C_{16}$.  Then $Y\supset C=C_1\cup\cdots\cup C_{16}\supset \varnothing$ is a stratification whose strata are Oka.
\end{lemma}     
   
\begin{proof}     
Since each curve $C_j\cong\P_1$ is Oka, we only need to prove that $Y\setminus C$ is Oka.  To see this, note that the involution $\tau: X\to X$ acts without fixed points on $X\setminus E$, so $X\setminus E$ is an unbranched double-sheeted covering space of $Y\setminus C$.  Now $X\setminus E$ is universally covered by $\C^2\setminus \widetilde\Gamma$.  Buzzard and Lu (\cite{Buzzard-Lu}, Proposition 4.1) showed that the discrete set $\widetilde\Gamma$ is tame in $\C^2$ in the sense of Rosay and Rudin \cite{Rosay-Rudin} (see also \cite{Forstneric2011}, Section 4.6).  Hence $\C^2\setminus \widetilde\Gamma$ is Oka (\cite{Forstneric2011}, Proposition 5.5.14).  Since the Oka property passes down along unbranched holomorphic covering maps (\cite{Forstneric2011}, Proposition 5.5.2), $X\setminus E$ and $Y\setminus C$ are also Oka.    
\end{proof}     

The surface $X$ obtained by blowing up a $2$-torus at finitely many points is Oka (\cite{Forstneric2011}, Corollary 6.4.12).  We do not know whether its quotient $Y$ is also Oka; the problem is that the quotient map $X\to Y$ is branched, and it is unknown whether the Oka property passes down along finite branched covering maps.  However, strong dominability of $X$ obviously passes down to give dominability of $Y$ at each point of $Y\setminus C$.  Since the projection $X\to Y$ is branched over $C$, holomorphic maps $\C^2\to Y$ dominating at points of $C$, which exist by Corollary \ref{c:Kummer-strongly-dominable}, do not factor through $X$. 

It would be interesting to know whether every Kummer surface $Y$ satisfies the Oka property restricted to maps from Stein surfaces to $Y$.  By inspecting the proof of Theorem~\ref{t:stratified-Oka}, we see that the problem is essentially topological.  Therefore it is not surprising that we have a positive result for Stein surfaces with sufficiently simple topology, that is, for the \textit{subcritical} Stein surfaces whose CW-decomposition does not contain any cells of index 2.

\begin{theorem}  \label{t:extension}
Let $Y$ be a Kummer surface with the rational curves $C=C_1\cup\cdots\cup C_{16}$.  Let $S$ be a Stein surface and $\Sigma$ be a smooth, possibly disconnected, complex curve in $S$.  If $S$ is obtained from $\Sigma$ by adding only cells of index $0$ or $1$, then every holomorphic map $f: \Sigma \to C$ extends to a holomorphic map $F: S\to Y$ such that $F(S\setminus \Sigma)\subset Y\setminus C$. If $f$ is dominating (of rank one) at some point $x\in \Sigma$, then the extension $F$ can be chosen to be dominating (of rank two) at $x$.
\end{theorem}

\begin{proof} 
By following the proof of Theorem~\ref{t:stratified-Oka}, we can extend $f$ to a holomorphic map $U\to Y$ from an open tubular neighbourhood $U$ of $\Sigma$ in $S$ such that $f(U\setminus \Sigma) \subset Y\setminus C$.  Furthermore, the extension can be chosen such that its rank at any point $x\in \Sigma$ equals the rank of $f\vert\Sigma$ plus $1$.  (It is important to observe that the normal bundle of a noncompact smooth complex curve is trivial. Indeed, any holomorphic vector bundle on an open Riemann surface is trivial by the Oka-Grauert principle; see \cite{Forstneric2011}, Theorem 5.3.1.)  Since $Y\setminus C$ is connected, the topological assumption on the pair $(S,\Sigma)$ implies that $f$ extends to a continuous map $S\to Y$ taking $S\setminus \Sigma$ to $Y\setminus C$.  Since $Y\setminus C$ is Oka by Lemma \ref{l:Kummer-are-stratified-Oka},  Theorem~\ref{t:Hayama} enables us to deform $f$ to a holomorphic map $F: S\to Y$ which agrees with $f$ to second order along $\Sigma$ and maps $S\setminus \Sigma$ to $Y\setminus C$.

However, since $Y\setminus C$ is not simply connected, it may be impossible to extend $f$ across cells of index $2$ in a relative CW-complex representing the pair $(S,\Sigma)$. Indeed, if the restriction of the given map to the boundary of a cell of index $2$ represents a nontrivial loop in $Y\setminus C$, then the map cannot be extended across the cell (as a map into $Y\setminus C$). 
\end{proof}

\begin{example}
The surface $S=\C^2$ with $\Sigma$ equal to the union of finitely many mutually disjoint affine complex lines $\Sigma_1,\ldots,\Sigma_m$ satisfies the hypothesis of Theorem~\ref{t:extension}.  Hence there is a dominating holomorphic map $\C^2\to Y$ whose restriction to each $\Sigma_j$ equals any prescribed biholomorphism onto the complement of a chosen point in any of the rational curves $C_1,\ldots,C_{16}$.
\end{example}


\section{Oka surfaces of class VII}  \label{s:class-VII}

\noindent
In this section we prove Theorem~\ref{t:class-VII} by considering each of the five classes of minimal surfaces $X$ of class VII in turn.

If $X$ is Hopf, then the universal covering space of $X$ is $\C^2\setminus\{0\}$, which is Oka, so $X$ is Oka.  

If $X$ is Inoue, then the universal covering space of $X$ is $\D\times\C$, where $\D$ denotes the open disc, which clearly carries a nonconstant negative plurisubharmonic function. Hence an Inoue surface is not strongly Liouville. 

If $X$ is intermediate, then $X$ is not strongly Liouville by \cite{Dloussky-Oeljeklaus}, Corollary 2.13.

Let $X$ be Inoue-Hirzebruch and $D$ be the union of the finitely many rational curves in $X$.  Let $\tilde D$ be the preimage of $D$ in the universal covering space $\widetilde X$ of $X$.  The complement $\widetilde X\setminus\tilde D$ is described in \cite{Dloussky-Oeljeklaus}, proof of Theorem 2.16, and in \cite{Zaffran}, pp.\ 400--401.  In the notation of \cite{Dloussky-Oeljeklaus}, $\widetilde X\setminus\tilde D$ is isomorphic to the image by the map $\C\times\C\to\C^*\times\C^*$, $(\zeta_1,\zeta_2)\mapsto (e^{\zeta_1},e^{\zeta_2})$, of the half-space in $\C\times\C$ defined by the inequality $-d\Re\zeta_1+c\Re\zeta_2<0$, where $d<0<c$.  Thus, $-d\log\lvert z_1\rvert+c\log\lvert z_2\rvert$ defines a nonconstant negative plurisubharmonic function on $\widetilde X\setminus\tilde D$, which extends across $\tilde D$ to a plurisubharmonic function on $\widetilde X$, so $X$ is not strongly Liouville.

It is easily seen that not being strongly Liouville is preserved by blowing up. It follows that blown-up Inoue, Inoue-Hirzebruch, and intermediate surfaces are not strongly Liouville.

By Enoki's original construction of the surfaces that now bear his name (\cite{Enoki1980}; \cite{Enoki1981}, Section 3), the universal covering space $Y$ of an Enoki surface $X$ is obtained as follows.  Let $W_0=\P_1\times\C$ and $\Gamma=\{\infty\}\times\C\subset W_0$.  For each $k\geq 0$, $W_{k+1}$ is $W_k$ blown up at two distinct points $p_k$ and $p_{-k-1}$, such that, when $k\geq 1$, $p_k$ lies in the total transform of $p_{k-1}$, and $p_{-k-1}$ lies in the total transform of $p_{-k}$.  We take $p_0=(\infty,0)$ and $p_{-1}=(a,0)$ with $a\in \C$.  Also,  $p_k$ lies in the proper transform $\Gamma_k$ of $\Gamma$, but $p_{-k-1}$ lies outside the total transform of $\Gamma$.  (We interpret the proper transform and the total transform of $\Gamma$ in $W_0$ as $\Gamma$ itself.)  Then $p_{-1},p_{-2},\ldots$ lie in the total transforms of the line $\{a\}\times\C$.  Let $Y_k=W_k\setminus(\Gamma_k\cup\{p_{-k-1}\})$.  Then $Y_k$ may be viewed as an open subset of $Y_{k+1}$, and $Y$ is the colimit of the sequence $Y_0\subset Y_1\subset Y_2\subset\cdots$.  (There is a misprint on p.\ 459 of \cite{Enoki1981}: the total transform of $p_{-k-1}$ is $C_{-k-1}$, not $C_{-k-2}$.)

If we formulate the Oka property as the convex approximation property, it is evident that if $Y_k$ is Oka for all $k\geq 0$, then $Y$ is Oka, so $X$ is Oka (\cite{Forstneric2011}, Proposition 5.5.6).  We claim that $W_k\setminus\Gamma_k$ is Zariski-locally affine (affine meaning algebraically isomorphic to $\C^2$); then $Y_k$ is Oka (\cite{Forstneric2011}, Proposition 6.4.6).

Being Zariski-locally affine is preserved by blowing up points (\cite{Gromov}, Section 3.5.D''; \cite{Forstneric2011}, Proposition 6.4.7).  Since $W_0\setminus\Gamma=\C\times\C$ is affine, the complement in $W_k$ of the total transform of $\Gamma$ is Zariski-locally affine.  Thus we need to show that every point in $W_{k+1}\setminus\Gamma_{k+1}$, $k\geq 0$, that lies in the total transform of $\Gamma$ has an affine Zariski-open neighbourhood in $W_{k+1}\setminus\Gamma_{k+1}$.

We claim that every point in $W_k$, $k\geq 0$, that lies in the total transform of $\Gamma$ has an affine Zariski-open neighbourhood $U$ in $W_k$ containing $p_k$ but not $p_{-k-1}$, in which $\Gamma_k$ appears as a straight line.  Namely, for $k=0$, let $U=(\P_1\setminus\{a\})\times\C$.  Suppose the claim is true for $k$ and let $w\in W_{k+1}$ lie in the total transform of $\Gamma$.  Let $V$ be an affine Zariski-open neighbourhood of the image of $w$ in $W_k$ containing $p_k$ but not $p_{-k-1}$, in which $\Gamma_k$ appears as a straight line.  Blowing up $V$ at $p_k$ yields a Zariski-open neighbourhood $V'$ of $w$ in $W_{k+1}$.  Take a line $L$ in $V$ through $p_k$ different from $\Gamma_k$, whose proper transform $L'$ contains neither $w$ nor $p_{k+1}$, and set $U=V'\setminus L'$.  Then $U$ is a Zariski-open neighbourhood of $w$ in $W_{k+1}$ containing $p_{k+1}$ but not $p_{-k-2}$.  Moreover, $U$ is algebraically isomorphic to the total space of an algebraic line bundle over $\C$, so $U$ is affine, and $\Gamma_{k+1}$ appears as a straight line in $U$.

Finally, let $w\in W_{k+1}\setminus\Gamma_{k+1}$, $k\geq 0$, lie in the total transform of $\Gamma$.  Let $U$ be an affine Zariski-open neighbourhood of the image of $w$ in $W_k$ containing $p_k$ but not $p_{-k-1}$, in which $\Gamma_k$ appears as a straight line.  Blowing up $U$ at $p_k$ and removing the proper transform of $\Gamma_k$ yields an affine Zariski-open neighbourhood of $w$ in $W_{k+1}\setminus\Gamma_{k+1}$.

This shows that minimal Enoki surfaces are Oka, so Theorem~\ref{t:class-VII} is proved.


\section{The Oka property is not closed in families}  \label{s:not-closed}

\noindent
Corollary~\ref{c:not-closed} follows from Theorem \ref{t:class-VII} because there is a family $\pi:Y\to\D$ of compact complex manifolds, that is, a proper holomorphic submersion and thus a smooth fibre bundle, such that the central fibre $\pi^{-1}(0)$ is an Inoue-Hirzebruch surface or an inter\-mediate surface, and the other fibres $\pi^{-1}(t)$, $t\in\D\setminus\{0\}$, are minimal Enoki surfaces.  See for instance the explicit examples on pp.\ 34--35 in \cite{Dloussky}.

Corollary~\ref{c:not-closed} implies that the Brody reparametrisation lemma that is used to show that Kobayashi hyperbolicity is open in families of compact complex manifolds \cite{Brody} has no higher-dimensional version that could be used to similarly prove that being the target of a nondegenerate holomorphic map from $\C^2$ is closed in families. (A weaker higher-dimensional version of the reparametrisation lemma was proved in \cite{Nguyen}.)

\begin{remark}  \label{r:C-connectedness}
A remark on the definition of $\C$-connectedness is in order.  It is not a well-known or much-studied property.  Gromov defined a complex manifold $X$ to be $\C$-connected if any two points of $X$ lie in the image of a holomorphic map $\C\to X$, that is, in an entire curve (\cite{Gromov}, Section 3.4).  We chose a weaker property in Definition~\ref{d:flexibility-properties}.  There are obvious alternatives, ranging from requiring every finite subset of $X$ to lie in an entire curve (this holds if $X$ is connected and Oka) to requiring that any two general points can be joined by a chain of entire curves.  We do not know whether these definitions are equivalent, but Corollary~\ref{c:not-closed} clearly holds for all of them.  It is of interest to compare $\C$-connectedness with rational connectedness, a well-understood property introduced in \cite{Campana} and \cite{Kollar-et-al}.  For a smooth proper algebraic variety, the four definitions we have mentioned, with entire curves replaced by rational curves, are equivalent (\cite{Kollar-et-al}, 2.1, 2.2).  Rational connectedness is deformation-invariant (\cite{Kollar-et-al}, 2.4), but by Corollary~\ref{c:not-closed}, $\C$-connectedness of compact complex manifolds is not.

In the context of $\C$-connectedness, we mention that Campana has introduced the notion of a compact K\"ahler manifold $X$ being \textit{special} \cite{Campana2004}.  It is an anti-hyperbolicity or anti-general type notion.  Campana has proved that $X$ is special if $X$ is either dominable, rationally connected, or has Kodaira dimension $0$.  On the other hand, if $X$ is of general type, then $X$ is not special.  Campana has conjectured that $X$ is special if and only if $X$ is $\C$-connected if and only if the Kobayashi pseudo-metric of $X$ vanishes identically.
\end{remark}


\section{Blowing an Oka manifold up or down}  \label{s:blowing-down}

\noindent
It is an open problem whether a blow-down of an Oka manifold is still Oka. In this section we prove the following partial result in this direction. 

\begin{theorem}  \label{t:blow-down}
Let $E$ be a discrete subset of a complex manifold $X$, and let $\widetilde X $ be $X$ blown up at each point of $E$.  Suppose that $\widetilde X$ has one of the following properties.
\begin{itemize}
\item[\rm (a)] $\widetilde X$ is an Oka manifold, and global holomorphic vector fields on $\widetilde X$ span the tangent space of $\widetilde X$ at each point.
\item[\rm (b)] $\widetilde X$ is elliptic in the sense of Gromov (see \cite{Gromov}, Section 0.5).
\end{itemize}
Then $X$ has the basic Oka property with approximation for maps $S\to X$ from Stein manifolds $S$ with $\dim S\leq\dim X$. 
\end{theorem}

The conclusion of the theorem means the following.  Given a Stein manifold $S$ of dimension at most $n=\dim X$, a holomorphically convex compact set $K$ in $S$, and a continuous map $f: S\to X$ that is holomorphic on an open neighbourhood $U$ of $K$ in $S$, it is possible to deform $f$ to a holomorphic map $f_1: S\to X$ through continuous maps $f_t: S\to X$, $t\in [0,1]$, that are holomorphic on an open set containing $K$ and are uniformly as close as desired to the initial map $f=f_0$ on $K$.

We recall that every elliptic manifold is Oka (\cite{Forstneric2011}, Corollary 5.5.12).  Note that we take a discrete subset to be closed by definition.

\begin{proof}
Assume first that $\dim S<n$. Pick a compact holomorphically convex set $K$ in $S$ and a continuous map $f: S\to X$ that is holomorphic on an open neighbourhood $U$ of $K$ in $S$. Choose a smaller open neighbourhood $U_1\Subset U$ of $K$. By the transversality theorem of Kaliman and Zaidenberg \cite{Kaliman-Zaidenberg} (see also \cite{Forstneric2011}, Section 7.8) there exists a holomorphic map $\tilde f\colon U_1\to X$, close to $f|U_1$, that is transverse to the discrete set $E \subset X$, and hence by dimension reasons it avoids $E$. Assuming as we may that $\tilde f$ is close enough to $f$ on $U_1$, and shrinking $U_1$ around $K$ if necessary, we conclude that $\tilde f|U_1$ is also homotopic to $f|U_1$ through a homotopy of holomorphic maps. By using a cut-off function in the parameter of the homotopy and the usual transversality theorem for smooth maps, we can extend $\tilde f$, without changing its values on a smaller neighbourhood $U_2\Subset U_1$ of $K$, to a smooth map $\tilde f\colon S\to X\setminus E$ that is holomorphic in a neighbourhood $U_2$ of $K$. Hence $\tilde f$ lifts to a continuous map $g: S\to \widetilde X$ that is holomorphic near $K$.  Since $\widetilde X$ is Oka, we can deform $g$ to a holomorphic map $S\to \widetilde X$ which approximates $g$ uniformly on $K$.  By pushing the homotopy down to $X$ we get a deformation of $\tilde f$, and hence of $f$, to a holomorphic map $f_1\colon S\to X$.

Suppose now that $\dim S=n$.  By the same perturbation argument as above we may assume that the restriction $f\vert U : U\to  X$ to an open neighbourhood $U$ of $K$ in $S$ is transverse to the discrete set $E\subset X$. Equivalently, every point of $E$ is a regular value of $f\vert U$.  This implies that the set $U\cap f^{-1}(E)$ is discrete, and after shrinking $U$ slightly around $K$ we may assume that it is finite.  

Choose a strongly plurisubharmonic Morse exhaustion function $\rho: S\to\mathbb R$, such that $K\subset \{\rho<0\} \Subset U$ and $0$ is a regular value of $\rho$.  Then $S$ is a cellular extension of the compact set $L=\{\rho\leq 0\}$, obtained by attaching to $L$ cells of index at most $n$.  

Write $f^{-1}(E)\cap L = \{p_1,p_2,\ldots, p_m\}$. By standard topological arguments we can deform $f$, keeping it fixed on $L$, to a new map $\tilde f: S\to X$ such that $\tilde f^{-1}(E)=\{p_1,p_2,\ldots, p_m\} \subset L$.  In fact, suppose that such a deformation has already been found over the sublevel set $\{\rho\le t\}$ for some $t\ge 0$, and we wish to extend it to $\{\rho\le t'\}$ for some $t'>t$.  We may assume that $t$ and $t'$ are regular values of $\rho$.  If $\rho$ has no critical values in $[t,t']$, there is no change of topology, so the extension obviously exists. At a critical point of $\rho$, we need to choose the extension of the deformation across a totally real disc of dimension equal to the Morse index of $\rho$ such that the range of the new map avoids $E$.  Since the real dimension of any such disc is at most $n$, the desired property can be achieved by a general position argument.  Thus, replacing $f$ by $\tilde f$, we may assume that $f^{-1}(E)=\{p_1,p_2,\ldots, p_m\} \subset L$.

Let $\widetilde S$ be obtained by blowing up $S$ at each of the points $p_1,\ldots, p_m$, and let $\tau : \widetilde S\to S$ be the blow-down map.  Then $\widetilde S$ is a $1$-convex manifold with the exceptional divisors $\Lambda_j\cong \P_{n-1}$ over the points $p_j$, and the map $\tau$ is the Remmert reduction of $\widetilde S$.  Similarly we denote by $\pi: \widetilde X\to X$ the blow-down map which collapses each of the exceptional divisors over the points of $E$.  Note that the restrictions $\pi: \widetilde X\setminus \pi^{-1}(E) \to X\setminus E$ and $\tau : \widetilde S \setminus \Lambda \to S\setminus \{p_1,\ldots,p_m\}$, where $\Lambda = \Lambda_1\cup\cdots\cup\Lambda_m$, are biholomorphic.

Since $f$ is locally biholomorphic near each of the points $p_1,\ldots,p_m$, it induces a holomorphic map $g: \widetilde S\to \widetilde X$ such that the following diagram commutes.
\[ \xymatrix{ 
	  \widetilde S  \ar[d]_{\tau}  \ar[r]^{g}  & \widetilde X \ar[d]^{\pi} \\ S \ar[r]^{f} & X}
\]
The map $g$ is holomorphic on a neighbourhood of $\widetilde L=\tau^{-1}(L)$ in $\widetilde S$ and maps each divisor $\Lambda_j$ biholomorphically onto the exceptional divisor $E_j=\pi^{-1}(f(p_j)) \subset \widetilde X$ over the image point $f(p_j)\in E$.  Furthermore, since $f$ maps $S\setminus \{p_1,\ldots,p_m\}$ to $X\setminus E$, its lifting $g$ maps $\widetilde S\setminus \Lambda$ to $\widetilde X\setminus \pi^{-1}(E)$. 

Since the manifold $\widetilde S$ is 1-convex, Theorem \ref{th:appendix} in the Appendix shows that the Oka principle applies to maps $\widetilde S\to \widetilde X$ that are holomorphic near the exceptional subvariety $\Lambda$ of $\widetilde S$, with interpolation to any given finite order along $\Lambda$, and with approximation on the holomorphically convex compact set $\widetilde L=\tau^{-1}(L) \subset \widetilde S$.  We thus obtain a holomorphic map $g_1 :  \widetilde S\to \widetilde X$, homotopic to $g$,  that agrees with $g$ to second order along $\Lambda$.  Hence $g_1$ descends to a holomorphic map $f_1: S\to X$: this is nontrivial only over a neighbourhood of the finite set $\{p_1,\ldots,p_m\}$ in $S$ that was blown up; elsewhere we get $f_1$ simply by noting that the restriction $\tau : \widetilde S \setminus \Lambda \to S\setminus \{p_1,\ldots,p_m\}$ is biholomorphic.
\end{proof}

Our next result presents a consequence of the hypothesis that the Oka property is preserved by blow-ups.  It may point to the hypothesis being false.  The first question to consider would be whether every embedded disc in $\C^3$ is a holomorphic retract of an embedded surface.

\begin{proposition}  \label{p:prediction}
Let $A$ be a submanifold of $\C^n$, contractible and not Oka.  If the blow-up of $\C^n$ along $A$ is Oka, then $A$ is a holomorphic retract of a hypersurface in $\C^n$.
\end{proposition}

\begin{proof}
Let $B$ be the blow-up of $\C^n$ along $A$ with the projection $\pi:B\to\C^n$.  The restriction $\pi: L =\pi^{-1}(A)\to A$ is the projectivised normal bundle of $A$, and $L$ is a smooth hypersurface in $B$.  

Since $A$ is contractible and Stein, its normal bundle is holomorphically trivial by the Oka-Grauert principle (\cite{Forstneric2011}, Section 5.3), so the restriction $\pi:L\to A$ admits a holomorphic section $\sigma: A\to L$.  Since $A$ and $\C^n$ are both contractible, the inclusion $A\hookrightarrow\C^n$ is an acyclic cofibration, so $A \stackrel{\sigma}\to L\hookrightarrow B$ extends to a continuous map $\tau: \C^n\to B$.  (Readers unfamiliar with the basic homotopy theory result invoked here may consult \cite{May}: it follows from the proposition on p.\ 44 that $A$ is a continuous retract of $\C^n$.)  Furthermore, as $B$ is Oka by assumption, we can choose  the extension $\tau$ to be holomorphic.  (Note that in the case of interest, when $\dim A\leq n-2$, $\tau$ is not a section of $\pi:B\to\C^n$, as $\pi$ has no continuous sections at all.) 

The preimage $H=\tau^{-1}(L)\subset\C^n$ is either all of $\C^n$ or a (possibly singular) hypersurface in $\C^n$.  Let $\rho=\pi\circ\tau:H\to A$.  Then $A\subset H$ and $\rho\vert A=\pi\circ\sigma=\textrm{id}_A$, so $\rho$ is a holomorphic retraction of $H$ onto $A$.  Now a holomorphic retract of an Oka manifold is Oka, so since $A$ is not Oka, $H$ is not $\C^n$.  
\end{proof}

\begin{remark}
By a more involved construction it is possible to insure that the hypersurface in Proposition \ref{p:prediction} is smooth (and drop the hypothesis that $A$ is not Oka).  Here is an outline; we leave the details to the reader.  

Using the notation in the proof of the proposition, let $N$ be the normal (line) bundle of the smooth hypersurface $L=\pi^{-1}(A)$ in $B$.  Choose a holomorphic section $\sigma : A\to L$ as in the proof.  Since $A$ is contractible, the restriction of the line bundle $N$ to the submanifold $\sigma(A) \subset L$ is holomorphically trivial, and thus admits a nowhere-vanishing holomorphic section $W$.  The normal bundle of $A$ in $\C^n$ is also trivial, so it admits a nowhere-vanishing holomorphic section $V$.  We can view $V$ as a holomorphic vector field on $\C^n$ along $A$.  Using the techniques in \cite{Forstneric2011}, Section 3.4, we can extend $\sigma: A\to L$ to a holomorphic map $\tilde \sigma : U\to B$ on an open neighbourhood $U$ of $A$ in $\C^n$, such that its differential $d\tilde \sigma$ at any point of $A$ maps $V$ to $W$.  This implies that $\tilde \sigma$ is transverse to $L$ in some open neighbourhood $U'\subset U$ of $A$.  By the same argument as above, using the Oka principle for maps $\C^n\to B$, we can interpolate $\tilde \sigma$ to second order along $A$ by a holomorphic map $\tau: \C^n\to B$ that is everywhere transverse to the hypersurface $L$ (see \cite{Forstneric2011}, Section 7.8).  It follows that the preimage $H=\tau^{-1}(L)$ is a smooth hypersurface in $\C^n$, and $A$ is a holomorphic retract of $H$.
\end{remark}

Finally, in case the reader is wondering what happens to Proposition \ref{p:prediction} when $A$ is Oka, we offer the following result.

\begin{proposition}
A contractible complex submanifold $A$ of $\C^n$ is Oka if and only if $A$ is a holomorphic deformation retract of $\C^n$ itself.
\end{proposition}

It is immaterial here whether or not the blow-up of $\C^n$ along $A$ is Oka.

It follows that every holomorphically embedded complex line $L$ in $\C^n$, $n\geq 2$, is a holomorphic deformation retract of $\C^n$, although $L$ need not be straightenable.  It is even possible that $\C^n\setminus L$ is Kobayashi hyperbolic (see \cite{Forstneric2011}, Section 4.18).

\begin{proof}
First, if $A$ is a holomorphic retract of $\C^n$, then $A$ is Oka.  Now suppose $A$ is Oka.  Since $A$ is also contractible, the identity map of $A$ extends by the inclusion $A\hookrightarrow \C^n$ to a holomorphic retraction $\rho:\C^n\to A$. 

Let $W=(\C^n\times\{0,1\})\cup(A\times[0,1])$ and define a continuous map $f:W\to\C^n$ by $f(x,0)=x$ and $f(x,1)=\rho(x)$ if $x\in\C^n$, and $f(a,t)=a$ if $a \in A$ and $t\in[0,1]$.  Since $W\hookrightarrow\C^n\times[0,1]$ is a cofibration and $\C^n$ is contractible, $f$ extends to a continuous map $g:\C^n\times[0,1]\to\C^n$.  Since $\C^n$ is Oka, we can use the parametric Oka principle (\cite{Forstneric2011}, Theorem 5.4.4) to deform $g$, keeping it fixed on $W$, to a continuous map $h$ such that $h(\cdot,t):\C^n\to\C^n$ is holomorphic for all $t\in[0,1]$.  This shows that $A$ is a holomorphic deformation retract of $\C^n$.
\end{proof}


\section{Appendix: The Oka principle for maps from 1-convex manifolds}  \label{appendix}

In this appendix we sketch the proof of the following result which was used in the proof of Theorem \ref{t:blow-down}. 

Recall that a complex manifold $S$ is said to be 1-{\em convex} if $S$ admits an exhaustion function that is strongly plurisubharmonic outside a compact subset of $S$. Such a manifold $S$ contains a maximal compact complex subvariety $\Lambda$ of positive dimension (possibly empty). The Remmert reduction of $S$, which is obtained by blowing down each connected component of $\Lambda$ to a point, is a (possibly singular) Stein space.

\begin{theorem}
\label{th:appendix}
Let $X$ be a complex manifold with one of the following properties.
\begin{itemize}
\item[\rm (a)] $X$ is an Oka manifold, and holomorphic vector fields on $X$ span the tangent space of $X$ at each point.
\item[\rm (b)] $X$ is elliptic in the sense of Gromov (see \cite{Gromov}, Section 0.5).
\end{itemize}
Given a 1-convex manifold $S$ with the exceptional variety $\Lambda$, a compact  holomorphically convex subset $L$ of $S$, and a continuous map $f:S\to X$ which is holomorphic in a neighbourhood of $\Lambda\cup L$, there is a holomorphic map $\tilde f:S\to X$ that agrees with $f$ to a given finite order along $\Lambda$ and uniformly approximates $f$ as closely as desired on $L$. The map $\tilde f$ can be chosen to be homotopic to $f$ through a homotopy of maps that are holomorphic near $\Lambda\cup L$, agree with $f$ to a given finite order along $\Lambda$, and uniformly approximate $f$ on $L$.
\end{theorem}

This \textit{relative Oka principle} is due to Henkin and Leiterer \cite{Henkin-Leiterer} in the classical case when $X$ is complex homogeneous. 

In \cite{Prezelj}, J.\ Prezelj stated a much more general Oka principle for sections of holomorphic submersions $Z \to S$ onto 1-convex manifolds $S$.  A map $S\to X$ can be identified with a section of the trivial submersion $Z=S\times X  \to S$, and in this case the only hypothesis in her result (Theorem 1.1 in \cite{Prezelj}) is that the fibre $X$ of the submersion is an Oka manifold. However, there is a gap in the proof given in \cite{Prezelj}, and it is not clear at this time whether the proof can be repaired. The (only) problem lies in the construction of a local dominating holomorphic spray around a given holomorphic section $f:S\to Z$ (see Section 4.2 in \cite{Prezelj}).  We lack a proof that there exist vertical holomorphic vector fields in certain conical Stein neighbourhoods of the graph $f(S\setminus \Lambda)\subset Z$ that are bounded near $f(\Lambda)$ (here, $\Lambda \subset S$ is the exceptional variety) and that generate the vertical tangent bundle of $Z$. In the special case of interest to us, the proof in \cite{Prezelj} is correct and complete provided that $X$ is an Oka manifold satisfying the following additional property.

\smallskip
\noindent {\bf Property ELS (existence of local sprays).}  A complex manifold $X$ satisfies ELS if for every holomorphic map $f: S\to X$ from a 1-convex manifold $S$ and every relatively compact open set $U\Subset S$ containing the exceptional variety $\Lambda$ of $S$, there exist an integer $N\in \mathbb N$, an open set $V\subset \C^N$ with $0\in V$, and a holomorphic map $F:U \times V \to X$ satisfying the following properties:
\begin{itemize}
\item[\rm (a)] $F(s,0)=f(s)$ for all $s\in U$, 
\item[\rm (b)] $F(s,t)=f(s)$ for all $s\in \Lambda$ and $t\in V$, and
\item[\rm (c)] the partial differential
\[ \frac{\partial}{\partial t}\bigg|_{t=0} F(s,t) : \C^N \longrightarrow T_{f(s)} X \]
is surjective for every $s \in U \setminus \Lambda$.  
\end{itemize}
\smallskip

Under the additional hypotheses that an Oka manifold $X$ satisfies ELS, the proof of the Oka principle for maps $S\to X$ from $1$-convex manifolds requires only minor modifications of the proof in the standard case when $S$ is a Stein manifold; see \cite{Prezelj} and Chapter 5 in \cite{Forstneric2011}. In fact, it suffices to postulate the existence of a holomorphic spray as in the definition of ELS merely in a small open neighbourhood of the exceptional variety $\Lambda$; using standard tools (see for example \cite{Forstneric2007}) we can then construct a spray with the stated properties over any relatively compact open subset of $S$. 

Therefore, to complete the proof of Theorem \ref{th:appendix}, we need the following lemma.

\begin{lemma}
{\rm (i)}  If global holomorphic vector fields on a complex manifold $X$ span the tangent space of $X$ at each point, then 
$X$ satisfies ELS.

{\rm (ii)}  Every elliptic manifold satisfies ELS.
\end{lemma} 

\begin{proof}[Proof of {\rm (i)}] Fix a 1-convex manifold $S$, a holomorphic map $f:S\to X$, and an open relatively compact subset $U\Subset S$ containing the exceptional subvariety $\Lambda$ of $S$. Since the set $f(\overline U) \subset X$ is compact, the condition on $X$ gives finitely many holomorphic vector fields $V_1,\ldots,V_N$ whose values at any point $x\in f(\overline U)$ span the tangent space $T_x X$. Let $\phi^j_t$ denote the flow of $V_j$. The formula
\begin{equation}
\label{eq:spray}
	F(s,t_1,\ldots,t_N) = \phi^1_{t_1}\circ \cdots \circ \phi^N_{t_N} (f(s)) \in X,\quad s\in S,
\end{equation}
defines a holomorphic map in an open neighbourhood of $S\times \{0\}$ in $S\times \C^N$. Clearly we have $F(s,0)=f(s)$ for every $s\in S$, and  
\[ \frac{\partial}{\partial t_j}\bigg|_{t=0} F(s,t) = V_j(f(s))\in  T_{f(s)} X \]
for every $s\in S$ and $j=1,\ldots, N$. Our choice of the vector fields $V_j$ implies that $F$ is a dominating holomorphic spray over $U$. 

To get sprays that are independent of the parameter $t$ over the points $s\in \Lambda$ (see property (b) in the definition of ELS), we choose finitely many holomorphic functions $g_1,\ldots, g_l$ on $U$ with $\{g_1,\ldots,g_l=0\} =\Lambda$, and replace each term $\phi^j_{t_j}$ in \eqref{eq:spray} by the composition of $l$ terms $\phi^j_{t_{j,k}g_k(s)}$, $k=1,\ldots,l$, where $t_j=(t_{j,1},\ldots,t_{j,l}) \in\C^l$ is near the origin.  (There is no misprint here: the function $g_k(s)$ appears as a multiplicative factor in the time variable of the flow.)  Since $g_{k}(s)=0$ for $s\in \Lambda$, each of the maps $\phi^j_{t_{j,k}g_k(s)}$ agrees with the identity when $s\in \Lambda$.  If on the other hand $s\in U\setminus \Lambda$, then $g_k(s)\ne 0$ for some $k$, and the composition of flows $\phi^j_{t_{j,k}g_k(s)}$ for $j=1,\ldots,N$ (and with fixed $k$) is a spray which is dominating at $s$.

\smallskip\noindent {\em Proof of {\rm (ii)}.} 
Let $(E,\pi,\sigma)$ be a dominating spray on $X$; that is, $\pi: E\to X$ is a holomorphic vector bundle and $\sigma:E\to X$  is a holomorphic map such that $\sigma(0_x)=x$ and the differential $d\sigma_{0_x} : T_{0_x}E \to T_x X$ maps the subspace $E_x$ of $T_{0_x}E$ surjectively onto $T_x X$ for every $x\in X$. (Here, $0_x$ is the zero element of the fibre $E_x$.) We may consider $(E,\pi,\sigma)$ as a fibre-dominating spray on the product submersion $S\times X\to S$ which is independent of the base variable $s\in S$. Let $f:S\to X$ be a holomorphic map. We associate to $f$ the section $\tilde f(s)=(s,f(s))$  of $S\times X \to S$. The restriction of $E$ to the graph $\widetilde S = \tilde f(S)\subset S\times X$ of $f$ is a holomorphic vector bundle over the 1-convex manifold $\widetilde S$. By a standard result (see e.g.\ \cite{Prezelj}, Theorem 2.4), such a bundle admits finitely many holomorphic sections $V_1,\ldots, V_N$ that vanish on the exceptional variety $\widetilde \Lambda =\tilde f(\Lambda)$ and generate the fibre $E_z$ over each point $z=\tilde f(s) \in \widetilde S\setminus \widetilde \Lambda$. The holomorphic map $F:S\times\C^N\to X$, defined by
\[
	F(s,t_1,\ldots, t_N)= \sigma\bigl( t_1V_1(\tilde f(s)) + \cdots + t_N V_N(\tilde f(s)) \bigr),
\]
is then a globally defined dominating holomorphic spray of maps $S\to X$ with the core map $f=F(\cdotp,0)$. As in the proof of (i) above, we can also construct sprays that are fixed over the exceptional variety.  This shows that $X$ satisfies ELS.   
\end{proof}

\end{document}